 \documentclass[11pt]{amsart}
\usepackage{amsfonts,amsmath,amsthm}
\begin{document}

 \newtheorem{thm}{Theorem}[section]
 \newtheorem{cor}[thm]{Corollary}
 \newtheorem{lem}[thm]{Lemma}{\rm}
 \newtheorem{prop}[thm]{Proposition}

 \newtheorem{defn}[thm]{Definition}{\rm}
 \newtheorem{assumption}[thm]{Assumption}
 \newtheorem{rem}[thm]{Remark}
 \newtheorem{ex}{Example}
\numberwithin{equation}{section}

\def\x{\mathbf{x}}
\def\P{\mathbb{P}}
\def\h{\mathbf{h}}
\def\p{\mathbf{p}}
\def\by{\mathbf{y}}
\def\bz{\mathbf{z}}
\def\F{\mathcal{F}}
\def\R{\mathbb{R}}
\def\T{\mathbf{T}}
\def\N{\mathbb{N}}
\def\D{\mathbf{D}}
\def\V{\mathbf{V}}
\def\U{\mathbf{U}}
\def\K{\mathbf{K}}
\def\Q{\mathbf{Q}}
\def\M{\mathbf{M}}
\def\oM{\overline{\mathbf{M}}}
\def\O{\mathbf{O}}
\def\C{\mathbb{C}}
\def\P{\mathbf{P}}
\def\Z{\mathbb{Z}}
\def\H{\mathcal{H}}
\def\A{\mathbf{A}}
\def\V{\mathbf{V}}
\def\AA{\overline{\mathbf{A}}}
\def\B{\mathbf{B}}
\def\c{\mathbf{C}}
\def\p{\mathbf{p}}
\def\L{\mathcal{L}}
\def\bS{\mathbf{S}}
\def\H{\mathcal{H}}
\def\I{\mathbf{I}}
\def\Y{\mathbf{Y}}
\def\X{\mathbf{X}}
\def\G{\mathbf{G}}
\def\f{\mathbf{f}}
\def\z{\mathbf{z}}
\def\v{\mathbf{v}}
\def\y{\mathbf{y}}
\def\d{\hat{d}}
\def\bx{\mathbf{x}}
\def\bI{\mathbf{I}}
\def\y{\mathbf{y}}
\def\g{\mathbf{g}}
\def\w{\mathbf{w}}
\def\b{\mathbf{b}}
\def\a{\mathbf{a}}
\def\u{\mathbf{u}}
\def\q{\mathbf{q}}
\def\e{\mathbf{e}}
\def\s{\mathcal{S}}
\def\cc{\mathcal{C}}
\def\co{{\rm co}\,}
\def\tg{\tilde{g}}
\def\tx{\tilde{\x}}
\def\tg{\tilde{g}}
\def\tA{\tilde{\A}}
\def\tf{\tilde{f}}
\def\bS{\mathbb{S}}
\def\sonc{{\bf SONC}}
\def\sosc{{\bf SOSC}}
\def\fonc{{\bf FONC}}

\title[Forms and spurious local minima]{
Homogeneous polynomials and spurious local minima on the unit sphere}
\author{Jean B. Lasserre}
\thanks{Work partly funded by the AI Interdisciplinary Institute ANITI through the French ``Investing for the Future PI3A" program under the Grant agreement ANR-19-PI3A-0004}
\address{LAAS-CNRS and Institute of Mathematics\\
University of Toulouse\\
LAAS, 7 avenue du Colonel Roche\\
31077 Toulouse C\'edex 4, France\\
email: lasserre@laas.fr}

\date{}

\begin{abstract}
We consider  forms
on the Euclidean unit sphere.  
We obtain obtain a simple and complete characterization of all 
points that satisfies the standard second-order necessary condition of optimality. It is stated
solely in terms of the value of (i) $f$, (ii) the norm of its gradient, and (iii) the first two smallest eigenvalues of its Hessian, all evaluated at the point. In fact this property also holds for twice continuous differentiable functions that are positively homogeneous.
We also characterize a class of degree-$d$ forms with no spurious local minima
on $\bS^{n-1}$ by using  a property of gradient ideals in algebraic geometry.
\end{abstract}

\maketitle

\section{Introduction}
Let $\bS^{n-1}$ (resp. $\mathcal{E}_n$) denotes the unit sphere (resp. Euclidean unit ball) in $\R^n$,
and  consider the optimization problem
\begin{equation}
\label{def-pb}
f^*\,=\,\min_\x\,\{\,f(\x):\:\x\in\bS^{n-1}\,\}\,,
\end{equation}
where $f$ is a degree-$d$ form and $f^*$ is understood as the global  minimum. (For linear $f$ or
degree-$2$ forms, \eqref{def-pb} can be solved efficiently.)

\subsection*{Background} 
In large-scale optimization problems (as is typical in machine learning applications), so far only first-order methods (e.g. stochastic gradient and its variants)  can be implemented. Therefore in the quest of the global minimum
it is important to be able to escape spurious local minima (see e.g. works by Jin et al. \cite{ge1}) or identify and characterize cases where no spurious local minima exist (as e.g. in Ge et al. \cite{ge2}). See also the discussions in
\cite{ge1,ge2} and references therein. \\

Even though minimizing forms on the unit sphere is a quite specific problem, it has important applications
For instance:

- Finding the maximal cardinality of $\alpha(G)$ of a stable set in a graph  $G$ 
reduces to minimizing a cubic form on the unit sphere. 

- Deciding convexity of an $n$-variate form reduces
to minimizing a  form on $\bS^{2n-1}$.

- Deciding nonnegativity of an even degree form
reduces to minimizing this form on $\bS^{n-1}$.

- Deciding copositivity of a symmetric matrix reduces to check whether some associated quartic form is is nonnegative on $\R^n$ (equivalently on $\bS^{n-1}$).

- In quantum information, the \emph{Best Separable State problem} 
 also relates to homogeneous polynomial optimization; see e.g. \cite{fawzi}.\\
 
 Crucial in the above problems is the search for the \emph{global} optimum
 and if possible rates of convergence of specialized algorithms like e.g.,
  the Moment-SOS-hierarchy  \cite{doherty,fawzi}  for converging sequences
  of \emph{lower bounds}  and  another (different)  Moment-SOS-hierarchy for
  converging sequences of \emph{upper bounds} 
 described in Lasserre \cite{lass-siopt} with rates provided in de Klerk and Laurent \cite{laurent1}. For more details on applications of
 homogeneous optimization on the sphere, the interested reader is referred to the discussion in 
 Fang and Fawzi \cite{fawzi}, de Klerk and Laurent \cite{laurent1} and the references therein.\\
 
In this  paper, by restricting to optimization of \emph{forms} on the unit sphere, we  provide a complete and rather simple characterization of all points which satisfy first- and second-order optimality conditions, solely in terms of the norm of the gradient of $f$ and the first two smallest eigenvalues of its Hessian, which to the best of our knowledge seems to be new. Therefore all such points (and local minimizers
in particular) are characterized by some property of the spectrum off the Hessian; namely how its first two smallest eigenvalues relate to  the value of $f$, an algebraic property of the form. Indeed in the context \eqref{def-pb}, 
 convexity plays little if no role for the absence of spurious local minima. For instance,
an arbitrary \emph{quadratic} form $\x\mapsto f(\x):=\x^T\Q\x$ has always a unique local (hence global) minimum 
(the smallest eigenvalue of $\Q$) no matter if $f$ is convex or not.  

This simple characterization could help to understand  the ``no spurious local minima" situation. 
Then combining this characterization with a decomposition property of gradient ideals, 
one also obtains a sufficient condition that identifies a class of forms with no spurious local minima.

Moreover this characterization is also particularly useful for algorithmic purposes. 
Indeed it provides an easy practical test in  first- and/or second-order minimization algorithms,
to check whether  a current iterate can be a candidate local minimum. 
 
\subsection*{Contribution} We restrict \eqref{def-pb} to degree-$d$ forms, with $d>2$
since for $d\leq2$ the problem has an easy solution in closed form. Our contribution is two-fold:\\

$\bullet$ We first provide the following simple and complete characterization of standard first-order and second-order necessary
optimality conditions (respectively denoted by {\bf (FONC)} and {\bf (SONC)}).

If $\x^*\in\bS^{n-1}$ is a local minimizer then {\bf (FONC)}-{\bf (SONC)}) reads:
\begin{equation}
\label{a}
\Vert \nabla f(\x^*)\Vert\,=\,d\,\vert f(\x^*)\vert\quad\mbox{and}\: 
\left\{\begin{array}{rl}
\lambda_{1}(\nabla^2f(\x^*))&\geq\,\mbox{$d\,f(\x^*)$, if $f(\x^*)\geq0$,}\\
\lambda_{2}(\nabla^2f(\x^*))&\geq\,\mbox{$d\,f(\x^*)$, if $f(\x^*)<0$,}\end{array}\right.
\end{equation}
where $\nabla^2f(\x^*)$ is the Hessian of $f$ at $\x^*$ and $\lambda_1(\nabla^2f(\x^*))$
(resp. $\lambda_2(\nabla^2f(\x^*))$) denotes the smallest (resp. second smallest) eigenvalue of
$\nabla^2f(\x^*)$. 

Moreover, if $f(\x^*)<0$ then $f(\x^*)=\lambda_1(\nabla^2f(\x^*))/d(d-1)$,
and the second condition which also reads $\lambda_2(\nabla^2f(\x^*))\geq \lambda_1(\nabla^2f(\x^*))/(d-1)$, states that the second smallest eigenvalue should be sufficiently separated from the smallest one.

Notice that \eqref{a} is stated solely in terms of (i) the value of $f$, (ii) the norm of its gradient, and (iii) the first two smallest eigenvalues of its Hesssian, evaluated at the point $\x^*$. 
To the best of our knowledge this characterization appears to be new. It is also worth noticing that this characterization remains valid for functions that are
positively homogeneous (of degree $d$) and twice continuously differentiable, i.e., such that $f(\lambda\x)=\lambda^df(\x)$ for all
$\lambda>0$ and all $\x$.

Then checking whether a point $\x$ satisfies \sonc~ is 
remarkably simple. It reduces
to check \eqref{a}, i.e., check whether $\Vert\nabla f(\x)\Vert=d\,\vert f(\x)\vert$ and then compare the value
$f(\x)$ with the two smallest eigenvalues of the Hesssian. This is very useful for any local optimization algorithm since one can easily check whether a curent iterate satisfies \eqref{a}.

$\bullet$ Finally, with any degree-$d$ form $f$ we associate  a polynomial $g$ of degree $d$
such that (i) $g$ coincide with $f$ on $\bS^{n-1}$, and (ii) all points $\x\in\bS^{n-1}$ that satisfy {\bf (FONC)} are critical points of $g$ (i.e. $\nabla g(\x)=0$) and the converse is also true. Then by using the characterization
\eqref{a} and invoking a certain decomposition of gradient ideals already nicely exploited by Nie et al. \cite{nie2} for unconstrained optimization, we provide a characterization of a class of degree-$d$ forms with no spurious local minima on $\bS^{n-1}$. 

At last but not least, we also remark that if a form $f$ can take negative values then minimizing
$f$ on the (convex) Euclidean unit ball $\mathcal{E}_n$ is easier than on $\bS^{n-1}$ 
and yields same (negative) minima and minimizers. In this case one may 
adapt the previous result and characterize a larger class of degree-$d$ forms  
with no spurious \emph{negative} local minima on $\bS^{n-1}$.

\section{Homogeneous optimization on the sphere}

\subsection{Notation and preliminary results}
Let $\R[\x]$ denote the ring of polynomials in the variables
$\x=(x_1,\ldots,x_n)$ and let $\Sigma[\x]\subset\R[\x]$ be there space of
sums-of-squares polynomials (SOS). Denote by $\R[\x]_d\subset\R[\x]$ the space of polynomials of degree at most $d$. 
Let $\nabla f(\x)$ (resp. $\nabla^2f(\x)$) denote the gradient 
(resp. Hessian) of $f$ at $\x$.
Recall that given polynomials $g_1,\ldots,g_s\in\R[\x]$, 
the notation $I=\langle g_1,g_2,\ldots,g_s\rangle$
stands for the ideal
\[\{\:\sum_{j=1}^sh_j\,g_j\::\:h_j\in\kappa[\x]\:\}\,,\quad(\mbox{$\kappa=\R$ or $\C$)},\] 
of $\kappa[\x]$ generated by the polynomials $g_1,\ldots,g_m$.

A polynomial $f\in\R[\x]$ is homogeneous of degree $d$ (and called a form) if
$f(\lambda\,\x)\,=\,\lambda^d\,f(\x)$ for all $\x\in\R^n$ and all $\lambda\in\R$. Then the important Euler's identity states that $\langle \nabla f(\x),\x\rangle=d\,f(\x)$ for all $\x\in\R^n$.
Similarly, $\x\mapsto \nabla f(\x)$ is homogeneous of degree $d-1$ and so $\nabla^2f(\x^*)\x=(d-1)\,\nabla f(\x)$.

Given a polynomial $p\in\R[\x]_d$, its homogenization $\tilde{p}\in\R[x_0,\x]_d$ is defined by
\[(x_0,\x)\mapsto \tilde{p}(x_0,\x)\,:=\,x_0^d\,p(\x/x_0),\quad(x_0,\x)\in\R^{n+1}.\]

Given $n$ forms $f_1,\ldots,f_n\in\R[\x]$ with respective coefficient vectors
$\f_1,\ldots,\f_n$, and given the system of polynomial equations
\[f_1(\x)\,=\,\cdots\,=\,f_n(\x)\,=\,0,\]
the \emph{resultant} ${\rm Res}(f_1,f_2,\ldots,f_{n})\in\R[\f_1,\ldots,\f_n]$ is a homogeneous polynomial in
$(\f_1,\ldots,\f_n)$ with the property:
\begin{eqnarray}
\label{prop-res}
{\rm Res}(f_1,f_2,\ldots,f_{n})&=&0\quad\Leftrightarrow\\
 \exists\u\,(\neq0)\,\in\C^{n}:&&f_1(\u)=\cdots=f_{n}(\u)\,=\,0\,.\end{eqnarray}
See e.g. \cite{cox,sturmfels,nie2}.

For a real symmetric matrix $\A\in\R^{n\times n}$,
denote by $\lambda_1(\A)\leq\lambda_2(\A),\ldots\,\leq\lambda_n(\A)$, its eigenvalues arranged in increasing order.

\subsection*{Optimization on the Euclidean sphere}
A point $\x^*\in\bS^{n-1}$ is said to be a local minimizer (and $f(\x^*)$ a local minimum) if there exists $\varepsilon>0$ and
a ball $\B(\x^*,\varepsilon)=\{\x:\Vert\x-\x^*\Vert<\varepsilon\}$ such that
$f(\x^*)\leq f(\x)$ for all $\x\in\bS^{n-1}\cap\B(\x^*,\varepsilon)$.

Below we recall some standard results in optimization, concerned with necessary and/or sufficient 
for optimality, in the context of the optimization problem \eqref{def-pb};
for a detailed account see e.g. Bertsekas \cite{bertsekas}.

\begin{prop}
\label{prop-fonc-sonc}
Let $f\in\R[\x]$ and for every $\x\in\bS^{n-1}$,
let $\x^\perp:=\{\u\,\in\bS^{n-1}: \u^T\x=0\}$. If $\x^*\in\bS^{n-1}$ is local minimizer of \eqref{def-pb} then there exists $\lambda^*\in\R$ such that:

(i) The {\bf First-Order Necessary Optimality-Condition (FONC)} holds: 
\begin{equation}
\label{FONC}
\nabla f(\x^*)+2\lambda^*\x^*\,=\,0\,.
\end{equation}
(ii) The {\bf Second-Order Necessary Optimality-Condition (SONC)} holds: 
\begin{equation}
\label{SONC}
\u^T\nabla^2 f(\x^*)\u+2\lambda^*\,\geq\,0\,,\quad\forall \u\in(\x^*)^\perp.
\end{equation}
(iii) Conversely, if $\x^*\in\bS^{n-1}$ satisfies \eqref{FONC} and the
{\bf Second-Order Sufficiency Optimality-Condition (SOSC)}
\begin{equation}
\label{SOSC}
\u^T\nabla^2 f(\x^*)\u+2\lambda^*\,>\,0\,,\quad\forall \u\in(\x^*)^\perp,
\end{equation}
then $\x^*$ is a local minimizer of \eqref{def-pb}.
\end{prop}
\begin{proof}
At $\x^*\in\bS^{n-1}$ the gradient  of the constraint $\Vert\x\Vert^2=1$ at $\x^*$ is simply $2\x^*$ ($\neq0$) and 
therefore is linearly independent, i.e., a basic constraint qualification holds true. Therefore \eqref{FONC}-\eqref{SONC} and (iii) follow from standard results in non-linear programming \cite{bertsekas}. 
\end{proof}
The following result is  an easy consequence of Proposition \ref{prop-fonc-sonc} but 
useful for our purpose.
\begin{cor}
\label{equiv}
Let $f$ be a degree-$d$ form and $\x^*\in\bS^{n-1}$ be a local minimizer.
Then
in \eqref{FONC}, $2\lambda^*=-d\,f(\x^*)$. In addition, \eqref{FONC} holds if and only if 
\begin{equation}
\label{new-font}
\Vert\nabla f(\x^*)\Vert^2\,=\,d^2f(\x^*)^2\,,
\end{equation}
and {\bf (SONC)} reads:
\begin{equation}
\label{SONC-hom}
\u^T\nabla^2 f(\x^*)\u\,\geq\, d\,f(\x^*)\,,\quad\forall \u\in(\x^*)^\perp.
\end{equation}
\end{cor}
\begin{proof}
In \eqref{FONC} we obtain
\[d\,f(\x^*)\,=\,\langle\nabla f(\x^*),\x^*\rangle\,=\,-2\lambda^*\Vert\x^*\Vert^2\,=\,-2\lambda^*,\]
and therefore $\Vert\nabla f(\x^*)\Vert^2\,=\,(2\lambda^*)^2\,\Vert\x\Vert^2\,=\,d^2f(\x^*)^2$.
Then \eqref{SONC-hom} follows from \eqref{SONC}.
Conversely, assume that \eqref{new-font} holds at $\x^*\in\bS^{n-1}$. Then
\[\Vert \nabla f(\x^*)- d\,f(\x^*)\,\x^*\Vert^2\,=\,\Vert\nabla f(\x^*)\Vert^2
-\underbrace{2d\,f(\x^*) \langle \nabla f(\x^*),\x^*\rangle}_{=-2d^2f(\x^*)^2}
+d^2f(\x^*)^2\Vert\x^*\Vert^2,\]
that is,
\[\Vert \nabla f(\x^*)-d\,f(\x^*)\,\x^*\Vert^2\,=\,\Vert \nabla f(\x^*)\Vert^2-d^2f(\x^*)^2\,=\,0\,,\]
and so \eqref{FONC} holds with $\lambda^*=-d\,f(\x^*)/2$, and again \eqref{SONC-hom} follows from \eqref{SONC}.
\end{proof}
Note in passing that all  \fonc ~points are solutions of
\[\nabla f(\x)\,=\,d\,f(\x)\,\x\,,\]
a system of $n$ polynomial equations in $n$ variables (the dual variable $\lambda^*$ in \eqref{SONC}  has been identified, thanks to
Euler' identity). Then generically, by Bezout's theorem it has at most $(d+1)^n$ solutions. 

\subsection{A distinguished representation}

In this section we obtain a more specific characterization of points
that satisfies {\bf (FONC)}-{\bf (SONC)} solely in terms of 
$f(\x^*)$, $\lambda_1(\nabla^2f(\x^*))$ and $\lambda_2(\nabla^2f(\x^*))$.

When $d\leq2$, Problem \eqref{def-pb} is easy and completely solved
analytically so we only consider the case $d>2$.

\begin{lem}
\label{lem-tau}
Let $f\in\R[\x]$ be a form of degree $d>2$, and let
$\x^*\in\bS^{n-1}$ satisfy {\bf (FONC)}.
Define:
\begin{equation}
\label{tau}
\tau(\x^*)\,:=\,\min_{\u\in(\x^*)^\perp}\u^T\nabla^2f(\x^*)\u\,.
\end{equation}
Then
\begin{equation}
\label{carac-eig}
\lambda_1(\nabla^2f(\x^*))\,=\,\min\,[\,d\,(d-1)\,f(\x^*)\,,\,\tau(\x^*)\,]\,,
\end{equation}
and if $\lambda_1(\nabla^2f(\x^*))=d\,(d-1)\,f(\x^*)$ then $\tau(\x^*)=\lambda_2(\nabla^2f(\x^*))$. \\

\noindent
Hence if $\tau(\x^*)\neq\lambda_1(\nabla^2f(\x^*))$ then
$\lambda_1(\nabla^2f(\x^*))=d\,(d-1)\,f(\x^*)$ and $\tau(\x^*)=\lambda_2(\nabla^2f(\x^*))$.
\end{lem}
\begin{proof}
Observe that $\R^n=\theta\x^*\oplus \gamma\,(\x^*)^\perp$
where $\theta,\gamma$ runs over $\R$.
Then writing $\v\in\bS^{n-1}$ as $\theta\x^*+\gamma\u$ with $\u\in(\x^*)^\perp$, 
one obtains $\Vert\v\Vert^2=\theta^2+\gamma^2$. Next,
\[ \v^T\nabla^2f(\x^*)\v\,=\,\theta^2\langle \x^*,\nabla^2f(\x^*)\x^*\rangle
+2\gamma\theta\,\langle\u,\nabla^2f(\x^*)\x^*\rangle+\gamma^2\,\u^T\nabla^2f(\x^*)\u.\]
Using homogeneity of $f$ (hence of $\nabla f(\x)$ as well), yields
\[\langle \x^*,\nabla^2f(\x^*)\x^*\rangle\,=\,
(d-1)\langle\x^*,\nabla f(\x^*)\rangle
\,=\,d\,(d-1)\,f(\x^*),\]
and
\[\langle \u,\nabla^2f(\x^*)\x^*\rangle\,=\,(d-1)\langle\u,\nabla f(\x^*)\rangle\,=\,
d\,(d-1)\,f(\x^*)\,\u^T\x^*=0,\]
so that
\[\v^T\nabla^2f(\x^*)\v\,=\,\theta^2\,d(d-1)\,f(\x^*)
+\gamma^2\langle \u,\nabla^2f(\x^*)\u\rangle\,.\]
This yields
\[\lambda_1(\nabla^2f(\x^*))
\,=\,\min_{\Vert\v\Vert=1}\v^T\nabla^2f(\x^*)\v\,=\,\min\,[\,d\,(d-1)\,f(\x^*)\,,\tau(\x^*)\,],\]
which is the desired result \eqref{carac-eig}.
Next, if $\lambda_1(\nabla^2f(\x^*))\,=\,d\,(d-1)\,f(\x^*)$ 
(hence with associated eigenvector $\x^*$), then
\[\lambda_2(\nabla^2f(\x^*))\,=\,\min_{\v\perp\x^*\,;\Vert\v\Vert=1}\v^T\nabla^2f(\x^*)\v\,=\,\tau(\x^*).\]
Conversely, if $\tau(\x^*)=\lambda_2(\nabla^2f(\x^*))>\lambda_1(\nabla^2f(\x^*))$
then by \eqref{carac-eig}, $\lambda_1(\nabla^2f(\x^*))=d(d-1)\,f(\x^*)$.
\end{proof}
We are now in position to characterizes in a simple compact form,
all points of $\bS^{n-1}$ that satisfy {\bf (SONC)}
when $f$ is a degree-$d$ form.
\begin{cor}
\label{nc-hom}
Let $f$ be a degree-$d$ form with $d>2$, and let $\x^*\in\bS^{n-1}$ satisfy {\bf (FONC)}.
Then $\x^*$ satisfies {\bf (SONC)} if and only if:
\begin{eqnarray}
\label{lambdamin1}
\lambda_1(\nabla^2f(\x^*))&\geq&d\,f(\x^*)\,\quad\mbox{if $f(\x^*)\geq0$}\\
\label{lambdamin2}
\lambda_2(\nabla^2f(\x^*))&\geq&d\,f(\x^*)\,\quad\mbox{if $f(\x^*)<0$}\,.
\end{eqnarray}
Moreover, if $f(\x^*)<0$ then $\lambda_1(\nabla^2f(\x^*))=d(d-1)f(\x^*)$.

If $d=2$ then $\x^*$ satisfies {\bf (SONC)} if and only if $\lambda_1(\nabla^2f(\x^*))\geq d\,f(\x^*)$
and there is only one local (hence global) minimum.
\end{cor}
\begin{proof}
i) $d>2$. First consider the case $f(\x^*)<0$.  By {\bf (SONC)}, 
$\tau(\x^*)\geq d\,f(\x^*)>d(d-1)\,f(\x^*)$, and therefore by Lemma \ref{lem-tau}, 
$\lambda_1(\nabla^2f(\x^*))=d (d-1)\,f(\x^*)$ and $\lambda_2(\nabla^2(f(\x^*))=\tau(\x^*)\geq d\,f(\x^*)$.

Conversely, suppose that $\lambda_2(\nabla^2(f(\x^*))\geq d\,f(\x^*)$.
Then $\lambda_1(\nabla^2f(\x^*))=d(d-1) f(\x^*)$ because 
$d(d-1)f(\x^*)<d\,f(\x^*)\leq\lambda_2(\nabla^2f(\x^*))$ and $d(d-1)f(\x^*)$ is an eigenvalue.
Hence by Lemma \ref{lem-tau}, $\lambda_2(\nabla^2f(\x^*))=\tau(\x^*)\geq d\,f(\x^*)$, i.e., {\bf (SONC)} holds. 

Next, consider the case $f(\x^*)\geq0$. Then
{\bf (SONC)} $\Rightarrow$ \eqref{lambdamin1}
follows from Lemma \ref{lem-tau}. Indeed if
$\lambda_1(\nabla^2(f(\x^*))=d(d-1)f(\x^*)$ then
$\lambda_1(\nabla^2(f(\x^*))\geq df(\x^*)$, and if 
$\lambda_1(\nabla^2(f(\x^*))=\tau(\x^*)$ then
$\lambda_1(\nabla^2(f(\x^*))\geq df(\x^*)$ by {\bf (SONC)}.

\noindent
\eqref{lambdamin1}$\Rightarrow$ {\bf (SONC)}. 
Again by Lemma \ref{lem-tau}, $\tau(\x^*)\geq\lambda_1(\nabla^2f(\x^*))\geq df(\x^*)$, and therefore
{\bf (SONC)} holds.

ii) $d=2$. Then $d(d-1)=d$ and $f(\x)=\x^T\Q\x$ for some real matrix $\Q$. Then 
each point $\x^*$ that satisfies {\bf (FONC)} is an eigenvector of $\Q$ with associated 
eigenvalue $f(\x^*)\in\{\lambda_1,\lambda_2,\ldots,\lambda_n\}$ and $\nabla^2f(\x)=2\Q$ for all $\x$.
So let $\x^*$ satisfies {\bf (FONC)}. 

If $f(\x^*)=\lambda_j$ with $j>1$, then necessarily $\tau(\x^*)=d\lambda_1\leq df(\x^*)$
with equality only if $\lambda_k=\lambda_1$ for all $2\leq k\leq j$. Hence {\bf (SONC)} holds only if $f(\x^*)=\lambda_1$
and therefore $\lambda_1(\nabla^2f(\x^*))=d\lambda_1\geq df(\x^*)$. Conversely let $\lambda_1(\nabla^2f(\x^*))\,(=d\lambda_1)\geq df(\x^*)$
then necessarily $f(\x^*)=\lambda_1$ and {\bf (SONC)} holds because $\tau(\x^*)=d\lambda_2\geq d\lambda_1=df(\x^*)$.
\end{proof}

So Corollary \ref{nc-hom} states that in homogeneous optimization on the Euclidean sphere, first- and second-order necessary optimality conditions can be easily checked by inspection of the gradient and the first two smallest eigenvalues of the Hessian of $f$. In particular, if $\x^*\in\bS^{n-1}$ is a  local minimizer 
with $f(\x^*)<0$ then $f(\x^*)=\lambda_1(\nabla^2f(\x^*))/d(d-1)$ with $\x^*$ being the corresponding eigenvector of $\nabla^2f(\x^*)$.

\begin{rem}
It is worth noticing that the characterization of {\bf (FONC)} in \eqref{new-font} 
and {\bf (SONC)} in Corollary \ref{nc-hom} remains valid for twice continuously differentiable 
and positively homogeneous functions of degree $d$, that is, 
functions $f$ that satisfy $f(\lambda\x)\,=\,\lambda^d f(\x)$ for all $\lambda>0$ and all $\x\in\R^n$. Indeed nowhere in the proof we have used the fact that $f$ is a polynomial.
\end{rem}

\subsection{Minimizing on $\mathcal{E}_n$ rather than on $\bS^{n-1}$}

Notice that \eqref{FONC} (or equivalently \eqref{new-font}) also holds at a local maximum.

In this section we remark that if $f$ is a form, all non positive local minima of $f$ in \eqref{def-pb} 
are also local minima on $\mathcal{E}_n$. Conversely, \emph{all} local minima
$f^*$ on $\mathcal{E}_n$ are non positive (i.e., necessarily $f^*\leq0$) and are local minima
on $\bS^{n-1}$ (except if $f^*=0$  is attained only at $\x^*=0$);
hence in particular, no local maximum on $\mathcal{E}_n$ can be negative. 

So if $f$ can take negative values then it is definitely better and easier 
to minimize on $\mathcal{E}_n$ because $\mathcal{E}_n$ is a convex set.
In doing so one obtains a negative local minimum
and avoid any positive local minimum on $\bS^{n-1}$.

\begin{lem}
\label{equivalent}
Let $f\in\R[\x]$ be a form of degree $d$. Then:

(i) Every local minimum $f^*$ on $\mathcal{E}_n$ satisfies $f^*\leq0$.
If $f^*<0$ then it is attained at some $\x^*\in\bS^{n-1}$ and so $f^*$ 
is also a local minimum on $\bS^{n-1}$.
If $f^*=0$ then either $f^*$ is attained only at $\x^*=0$ or
$f^*$ is also a local minimum (also attained) on  $\bS^{n-1}$.

(ii) Every local minimum $f^*\leq0$ on $\bS^{n-1}$ is also a local minimum
on $\mathcal{E}_n$.
\end{lem}
\begin{proof}
(i) Assume that $f^*>0$ is a local minimum on $\mathcal{E}_n$
hence for some local minimizer $0\neq\x^*\in\mathcal{E}_n$. Then 
$\lambda\x^*\in\mathcal{E}_n$ for every $\lambda\in (0,1)$, and by homogeneity of $f$  one obtains $f(\lambda\x^*)=\lambda^df(\x)^*=\lambda^d f^*<f^*$, in contradiction with the hypothesis. Next, assume that $f^*<0$ and
$\x^*\in\mathcal{E}_n$ is a local minimizer (hence $\x^*\neq0$) with
$\Vert\x^*\Vert<1$. Then $\z^*:=\lambda\x^*\in\bS^{n-1}$ for some $\lambda>1$, and $f(\z^*)=\lambda^df(\x^*)<f^*$, a contradiction
and so necessarily $\x^*\in\bS^{n-1}$. If $f^*=0$ and $f^*$ is attained at $\x^*\neq0$, then $f^*=0$ is also attained
at $\z^*=\x^*/\Vert\x^*\Vert\in\bS^{n-1}$ and so is also a local minimum on $\bS^{n-1}$.

(ii) We proceed by contradiction. Assume $\x^*\in\bS^{n-1}$ is a local minimizer
of $f$ on $\bS^{n-1}$ with $f^*\leq0$ and not a local minimizer on $\mathcal{E}_n$. Let $\B_j(\x^*):=\{\y:\:\Vert \y-\x^*\Vert^2<1/j\}$.
Then for every integer $j>n_0$, there exists 
$\y_j\in\B_j(\x^*)\cap \mathcal{E}_n$
with $f(\y_j)<f(\x^*)\leq0$. Letting $\z_j:=\y_j/\Vert\y_j\Vert\in\bS^{n-1}$, one obtains $f(\z_j)=\Vert \y_j\Vert^{-d}f(\y_j)\leq f(\y_j)<f(\x^*)$. By letting $j$ increase one has exhibited a sequence $(\z_j)_{j\in\N}\subset\bS^{n-1}$ converging to 
$\x^*$ and with cost $f(\z_j)<f(\x^*)$ for all $j$, in contradiction with
our hypothesis.
\end{proof}
So if $f$ is homogeneous and not nonnegative on $\R^n$, then its global minimum $f^*$ 
on $\bS^{n-1}$ is  strictly negative. Then 
by Lemma \ref{equivalent}, searching for the global minimum
$f^*$ is equivalent to searching
for the global minimum of $f$ on the larger (but convex) set $\mathcal{E}_n$. 
Therefore consider the case where $f$ has 
no spurious negative local minima on $\bS^{n-1}$ (hence no
spurious negative local minima on $\mathcal{E}_n$) while spurious positive  local minima on $\bS^{n-1}$ may exist. In such a case, any local minimization algorithm on $\mathcal{E}_n$ (starting at $\x_0\in\mathcal{E}_n$ with $f(\x_0)<0$) converging to a Karush-Kuhn-Tucker point (i.e. a point that satisfies {\bf (FONC)}) will find the global minimum on $\mathcal{E}_n$ (and hence on $\bS^{n-1}$), and optimizing over $\mathcal{E}_n$ is certainly easier than on $\bS^{n-1}$.

\section{No spurious local minima on $\bS^{n-1}$}
\label{spurious}
In this section we are concerned with the ``\emph{no spurious local minima}" situation,
and characterize a class of degree-$d$ forms that have no
spurious local minima on $\bS^{n-1}$. 

\begin{cor}
\label{first-carac}
Let $f$ be a degree-$d$ form and 
\[\Theta\,:=\,\{\,\x\in\bS^{n-1}:\: \Vert\nabla f(\x)\Vert\,=\,d\,\vert f(\x)\vert\,\}\,,\]
i.e., $\Theta$ is the set of all points of $\bS^{n-1}$ that satisfy {\bf (FONC)}.

i) If $f$ is nonnegative then it  has no spurious local minima on $\bS^{n-1}$
if $f$ is constant on the set
\[\Delta^+\,:=\,\{\x\in\Theta\,: \: 
\lambda_1(\nabla^2f(\x)) \,\geq\,\Vert \nabla f(\x)\Vert\,\}\,,\]
in which case all \sonc ~ points are global minimizers.

ii) If $f$ can take negative values then it  has no spurious local minima on $\bS^{n-1}$
if 
\[\{\x\in\Theta\,: \: f(\x)\geq0\,;\: \lambda_1(\nabla^2f(\x)) \,\geq\,\Vert \nabla f(\x)\Vert\,\}\,=\,\emptyset\,,\]
and $f$ is constant on the set
\[\Delta^-\,:=\,\{\x\in\Theta\,: \: f(\x)\,<0\,;\: \lambda_2(\nabla^2f(\x))\,\geq\,-\Vert \nabla f(\x)\Vert\,\}\,.\]
\end{cor}
\begin{proof}
(i) By Corollary \ref{equiv}, $\Vert\nabla f(\x)\Vert=d\vert f(\x)\vert$ on $\Theta$. Next,  
if $f$ is constant on $\Delta^+$ and $f$ is nonnegative, then by Corollary \ref{nc-hom}, all points $\x^*\in\bS^{n-1}$ that satisfy \sonc~
(in particular all local minimizers) belong to $\Delta^+$. So if $f$ is constant on $\Delta^+$,
all \sonc ~ points (and local minimizers in particular) have same nonnegative value, and therefore they all are global minimizers.
A similar argument applies to prove (ii)
\end{proof}

Notice that Corollary \ref{first-carac} also holds for twice ontinuously differentiable
positively homogeneous functions.\\

We next show that the characterization in Corollary \ref{first-carac} is also related to
a property of gradient ideals of $\R[\x]$. We introduce  a polynomial
(not a form) with the following nice property. On $\bS^{n-1}$:

-  (i) it coincides with $f$ (up to a multiplicative constant), and 

-  (ii) all its critical points coincide with \fonc~points of $f$.

We then invoke a property of 
gradient ideals nicely exploited in Nie et al. \cite{nie2}.\\

Given a degree-$d$ form $f$, let $g\in\R[\x]_d$ be the polynomial
\begin{equation}
\label{g}\x\mapsto g(\x)\,:=\,f(\x)\,(1-\frac{d}{d+2}\Vert\x\Vert^2)\,,\quad\x\in\R^n\,.
\end{equation}
\begin{prop}
\label{step-ideal}
Let $f$ be a degree-$d$ form and let $g\in\R[\x]$ be as in \eqref{g}.
Then on $\bS^{n-1}$:
\begin{equation}
\label{g-fonc}
\nabla g(\x)\,=\,0\:\Leftrightarrow \nabla f(\x)\,=\,d\,f(\x)\cdot \x\:
\Leftrightarrow \Vert\nabla f(\x)\Vert^2\,=\,(d\,f(\x))^2\,.
\end{equation}
That is, all critical points of $g$ in $\bS^{n-1}$  satisfy {\bf (FONC)} for $f$, and conversely,
all points of $\bS^{n-1}$ that satisfy {\bf (FONC)} for $f$ are critical points of $g$.
\end{prop}
\begin{proof}
Observe that
\[\nabla g(\x)\,=\,\nabla f(\x)(1-\frac{d}{d+2}\Vert\x\Vert^2)-\frac{2d}{d+2}\,f(\x)\,\x\,.\]
Therefore if $\x\in\bS^{n-1}$ then $g(\x)=\frac{2}{d+2}f(\x)$ and ,
\[\nabla g(\x)\,=\,\frac{2}{d+2}\nabla f(\x) - \frac{2d}{d+2}\,f(\x)\cdot\x
\,=\,\frac{2}{d+2}\,(\nabla f(\x) - d\,f(\x)\,\x\,)\,,\]
and so $\nabla g(\x^*)=0$ if and only if \eqref{FONC} holds with $2\lambda^*=-d\,f(\x^*)$,
which yields the desired result \eqref{g-fonc}.
\end{proof}
Moreover, on $\bS^{n-1}$ minimizing $f$ is strictly equivalent to minimizing $g$ since 
$g(\x)=2 f(\x)/(d+2)$ on $\bS^{n-1}$. Next, with $g$ as in \eqref{g}, define  the \emph{gradient} ideal:
\[\mathcal{I}_{{\rm grad}}(g):=\langle \frac{\partial g(\x)}{x_1},\ldots,\frac{\partial g(\x)}{x_n}\rangle,\]
and its associated variety
\[V_{{\rm grad}}(g)\,:=\,V(\mathcal{I}_{{\rm grad}}(g))\,=\,\{\z\in\C^n:\:\nabla g(\z)\,=\,0\,\}.\]
Then $V_{{\rm grad}}(g)$ is a finite union of irreducible subvarieties $W_j$'s, that is,
\[V_{{\rm grad}}(g)\,=\,W_0\cup\,W_1\ldots \cup W_s\,,\]
with $W_0\cap\R^n=\emptyset$ and in addition,
$g$ is a real constant on each $W_j$, $j\geq1$; see e.g. \cite[\S2]{bochnak} and \cite[p. 592]{nie2}.  
So we can regroup all components on which $g$ takes the same value, and write
\begin{equation}
\label{var}
V_{{\rm grad}}(g)\,=\,W_0\cup\,\tilde{W}_1\ldots \cup \tilde{W}_r\,,
\end{equation}
where $g(\x)=g_j$ on $\tilde{W}_j$ and $g_j\neq g_i$ for all $1\leq i,j$ with $i\neq j$.

We are now in position to provide a characterization of 
a class of degree-$d$ forms $f$ with no spurious local minima on $\bS^{n-1}$.

\begin{thm}
\label{th-final}
Consider problem \eqref{def-pb} where $f$ is a degree-$d$ form ($d>2$).
Then $f$ has no spurious local minima on $\bS^{n-1}$ if 
there is only one index $j^*\geq1$ in \eqref{var} such that
\[\tilde{W}_{j^*}\,\cap\,\Omega\,\neq\,\emptyset\,,\]
where
\begin{eqnarray}
\label{def-omega}
\Omega &:=&\{\x\in\bS^{n-1}: f(\x)<0\,;\:\lambda_{2}(\nabla^2f(\x))\,\geq\,-\Vert\nabla f(\x)\Vert\,\}
\\
\label{final}
&&\cup \quad\{\x\in\bS^{n-1}: f(\x)\geq0\,;\:\lambda_{1}(\nabla^2f(\x))\,\geq\,\Vert\nabla f(\x)\Vert\,\}\,.
\end{eqnarray}
If $f$ has no spurious local minima on $\bS^{n-1}$ then all  \sonc~ points 
with value $f^*$ (i.e., all local hence global minimizers) belong to a unique set
$\tilde{W}_{j^*}\,\cap\,\Omega$. The  other nonempty sets
$\tilde{W}_{j}\,\cap\,\Omega\neq\emptyset $ contain \sonc~points which are not local minimizers.
\end{thm}
\begin{proof}
 By construction each nonempty set  set $\tilde{W}_{j}\,\cap\,\Omega$ identifies
a subset of \sonc~points of $f$ which share the same value, say $f^*_j$. Therefore
$f^*_j\geq f^*$ for all $j$ (where $f^*$ is the global minimum).
 So if there only one such set $\tilde{W}_{j^*}\,\cap\,\Omega\neq\emptyset$ then necessarily
 $f^*_{j^*}=f^*$, otherwise a spurious local minimum $\tau>f^*$ would correspond to some
 \sonc ~point in another nonempty set  $\tilde{W}_{j}\,\cap\,\Omega$, with $j\neq j^*$, and $f^*_j=\tau$.
 
 With similar arguments, if $f$ has no spurious local minima on $\bS^{n-1}$
 then necessarily all local (hence global) minimizers are \sonc ~points and belong to the same set $\tilde{W}_{j^*}\,\cap\,\Omega$ for some index $j^*$.
 All other nonempty sets $\tilde{W}_{j}\,\cap\,\Omega\neq\emptyset $ contain \sonc~points which 
 cannot be local minimizers as their associated value $f^*_j\neq f^*$ must be larger than $f^*$.
   \end{proof}
We have seen that if a degree-$d$ form can take negative values then
all its local minima on $\mathcal{E}_n$ are negative local minima on $\bS^{n-1}$,
and the converse is true. Then for minimizing on $\mathcal{E}_n$,
it is interesting to characterize a class of degree-$d$ forms
with the less restrictive condition of no spurious \emph{negative} local minima on $\mathcal{E}_n$ (hence on $\bS^{n-1}$).

\begin{cor}
\label{cor-final}
Consider problem \eqref{def-pb} where $f$ is a degree-$d$ form ($d>2$).
Then $f$ has no spurious negative local minima on $\bS^{n-1}$ if
there is only one index $j^*\geq1$ in \eqref{var} such that
\[\tilde{W}_{j^*}\,\cap\,\Omega\,\neq\,\emptyset\,,\]
where $\Omega=\{\,\x\in\bS^{n-1}: f(\x)<0\,;\:\lambda_{2}(\nabla^2f(\x))\,\geq\,-\Vert\nabla f(\x)\Vert\,\}$.
\end{cor}
The proof is similar to that of Theorem \ref{th-final}.
\begin{ex}
To illustrate Theorem \ref{th-final} and Corollary \ref{cor-final}, 
consider the following toy example with $n=2$, $d=3$ and $\x\mapsto f(\x):=x_1x_2^2$.
The polynomial $g$ in \eqref{g} reads $\x\mapsto g(\x)=x_1x_2^2-3\,x_1^3x_2^2/5-3\,x_1x_2^4/5$. Then:
\[\nabla g(\z)\,=\,0\Leftrightarrow\quad \begin{array}{ll}z_2^2\,(\,1-\,9\,z_1^2/5\,-\,3\,z_2^2/5\,) &\,=\,0\\
z_1\,z_2\,(\,2-6\,z_1^2/5\,-\,12\,z_2^2/5\,)&\,=\,0\end{array}\,.\]
So with $\tilde{W}_j$ as in \eqref{var}, we find that
$g$  is constant on the four subvarieties
\[\tilde{W}_1=\{(0,\pm \sqrt{5/3})\}\,:\quad
\tilde{W}_2=\{(x,0):x\in\R\}\,,\]
and 
\[\tilde{W}_3= \frac{1}{\sqrt{3}}\{(1,\pm\sqrt{2})\}\,;\quad
\tilde{W}_4= \frac{1}{\sqrt{3}}\{(-1,\pm\sqrt{2})\}\,.\]
with different values on each one of them.
Then 
\[\bS^{n-1}\cap \tilde{W_1}\,=\,\emptyset\,;\quad
\bS^{n-1}\cap \tilde{W_2}\,=\,\{(1,0)\}\,,\]
while
\[\bS^{n-1}\cap \tilde{W_3}\,=\,\{(\sqrt{1/3},\pm\sqrt{2/3})\}\,;\quad
\bS^{n-1}\cap \tilde{W_4}\,=\,\{(-\sqrt{1/3},\pm\sqrt{2/3})\}\,.\]
Next from $f(\x)=x_1x_2^2$,
\[\nabla^2f(\x)\,=\,\left[\begin{array}{cc}0&2x_2\\2x_2&2x_1\end{array}\right]\:\Rightarrow\:
\lambda_1(\nabla^2f(\x))\,=\,x_1-\sqrt{x_1^2+4x_2^2}\,.\]
Recall the definition of $\Omega$ in \eqref{def-omega}. 
Then $\Omega\cap\tilde{W}_4\neq\emptyset$. Indeed $\x^*=(-\sqrt{1/3},\pm\sqrt{2/3})\in\Omega$, with $f^*<0$,
because  \eqref{lambdamin2} holds; indeed:
\begin{eqnarray*}
\lambda_2(\nabla^2f(\x^*))&=&x_1^*+\sqrt{(x_1^*)^2+4(x_2^*)^2}\\
&=&-\sqrt{1/3}+3\sqrt{1/3}\,=\,2\sqrt{1/3}\,\geq\,d\, f(\x^*)=-2\sqrt{1/3}\,.\end{eqnarray*}
As $j=4$ is the only index for which $f(\x)<0$ in $\tilde{W}_j$, one concludes that $f$ has no spurious negative local minima on $\bS^{n-1}$ and 
$f^*=\frac{-2\sqrt{1/3}}{3}$ is the global minimum. 

On the other hand, 
$\tilde{W}_2\cap\Omega\neq0$. Indeed $\x^*=(1,0)\in\Omega$ with nonnegative value $f^*=0$
because \eqref{lambdamin1} holds, since
\[\lambda_1(\nabla^2f(\x^*))\,=\,x_1^*-\sqrt{(x_1^*)^2+4(x_2^*)^2}\,=\,1-1=0\,\geq\,d\,f(\x^*)=0\,.\]
 Finally $\tilde{W}_3\cap\Omega=\emptyset$  because $\x^*=(\sqrt{1/3},\pm\sqrt{2/3})\in\bS^{n-1}\cap\tilde{W}_3$
 does not satisfy \eqref{lambdamin1} (i.e., $\x^*$ is not a \sonc~point), since
\[\lambda_1(\nabla^2f(\x^*))\,=\,\sqrt{1/3}-3\sqrt{1/3}\,=\,-2\,\sqrt{1/3}\,\not\,\geq\,d\,f(\x^*)=2\,\sqrt{1/3}\,.\]

\end{ex}
\section{Conclusion}

In this paper we have considered homogeneous polynomial optimization on the 
Euclidean sphere $\bS^{n-1}$ and completely characterize all points 
that satisfy first- and second-order necessary optimality conditions, solely in terms of $f$, its gradient and the two smallest eigenvalues of its Hessian. Then one may characterize a class of degree-$d$ 
forms with no spurious local minima,
in particular  via some decomposition of a related gradient ideal.

While the characterization of all points 
that satisfy first- and second-order necessary optimality conditions
is also valid for twice continuously differentiable positively homogeneous functions, characterizing
a class of such functions with no spurious local minima 
is challenging as we cannot invoke algebraic properties of
$f$ any more; indeed the second characterization of no spurious local minima via a certain gradient variety 
is proper to forms. Another issue for further investigation is the case 
where $f$ is an arbitrary degree-$d$ polynomial and not a form any more.

\vspace{0.2cm}

\noindent
{\bf Acknowledgement:} The author gratefully acknowledges 
Professor Jiawang Nie (UCSD at San Diego) for fruitful discussions that helped improve the paper.

\end{document}